\documentclass{article}
\usepackage[utf8]{inputenc}
\usepackage[]{graphicx}
\usepackage{amsthm,amsmath,amssymb,tikz,tikz-cd,amsmath,mathrsfs,mathtools,multicol,dirtytalk,tabularx,xy,lipsum,url,enumitem,cmll}
\usepackage{float}
\usepackage{enumitem} 

\usepackage[total={6.5in, 8.5in}]{geometry}
\usepackage{blkarray}
\usepackage{tabularray}
\usepackage{array}
\usepackage[colorlinks=true,linkcolor=blue,citecolor=blue]{hyperref}
\usepackage{fancyhdr}
\usepackage{changepage}

\numberwithin{equation}{section}

\theoremstyle{plain}
\newtheorem{theorem}{Theorem}[section]
\newtheorem{lemma}[theorem]{Lemma}
\newtheorem{proposition}[theorem]{Proposition}
\newtheorem{corollary}[theorem]{Corollary}

\theoremstyle{definition}
\newtheorem{definition}[theorem]{Definition}
\newtheorem{remark}[theorem]{Remark}
\newtheorem{example}[theorem]{Example}

\DeclareMathOperator*{\T}{{\large\textbf{T}}}

\title{\textbf{Direct and ordinal products realized by triangular norm operators with no zero divisors}}
\author{Joseph McDonald\footnote{Email: jsmcdon1@ualberta.ca}\footnote{University of Alberta, Department of Philosophy, Edmonton, T6G 2E7, Canada}} 
\date{}
\begin{document}
\maketitle

\fancyhead{dfg}
\fancyhead[L]{Direct and ordinal products}
\fancyhead[C]{}
\fancyhead[R]{Joseph McDonald}

\begin{abstract}
In this note we continue the work of Chon, as well as Mezzomo, Bedregal, and Santiago, by studying algebraic operations on fuzzy posets and bounded fuzzy lattices. We first prove that fuzzy posets are closed under finite direct products whenever the triangular norm realizing the product construction has no zero divisors. This result is then extended to the case of bounded fuzzy lattices. Some immediate consequences are then obtained within the setting of direct products realized by triangular norms with no nilpotent elements as well as strictly monotone and cancellative triangular norms. We then introduce a triangular norm based construction of ordinal products and similarly show that fuzzy posets are closed under ordinal products whenever the triangular norm realizing the product construction has no zero divisors.             
\par
\vspace{.2cm}
  \end{abstract}

\maketitle

\section{Introduction}

The general aim of this study is to investigate certain generalizations of various concepts within the areas of universal and abstract algebra within the context of fuzzy set theory. Research along these general lines was initiated by Rosenfeld \cite{ros} who studied fuzzy groups, Liu \cite{liu} who studied fuzzy rings, and Kuroki \cite{kur} who studied fuzzy semigroups. Yuan and Wu \cite{yua} as well as Ajmal and Thomas \cite{ajm} introduced fuzzy lattices as algebras arising from certain lattice-valued fuzzy sets, i.e., fuzzy sets that have a lattice as the codomain of their associated membership/characteristic functions. Chon \cite{cho2009} alternatively defined fuzzy lattices as arising from fuzzy partially ordered sets and proved that fuzzy lattices are closed under direct minimum products in which the fuzzy partial order structure of the product construction is realized by the minimum triangular norm operator defined over the fuzzy partial order structure of each factor. Chon's findings generalize the well-known result within classical lattice theory which demonstrates that lattices are closed under direct products whereby the partial order structure of the product construction is defined coordinate-wise with respect to the partial order structure of each factor. Mezzomo, Bedregal, and Santiago \cite{mez} extended Chon's formulation of a fuzzy lattice in \cite{cho2009} to that of a bounded fuzzy lattice and proved that bounded fuzzy lattices are closed under taking direct minimum products, collapsed sums, liftings, and many other algebraic/order-theoretic constructions.

In this note we continue the work of Chon, as well as Mezzomo, Bedregal, and Santiago, by studying direct products of bounded fuzzy lattices. The first contribution of this study is to extend their result by showing that bounded fuzzy lattices are closed under a much more general construction of direct products; namely direct products that are defined using a triangular norm operator with no zero divisors.\footnote{We however work with a generalized formulation of transitivity on the underling fuzzy poset.} Some immediate consequences of this result are then obtained within the setting of direct products realized by triangular norms with no nilpotent elements as well as strictly monotone and cancellative triangular norms. We then introduce an appropriate generalization of ordinal products of posets to the setting of fuzzy posets and similarly show that this generalization mimicks the closure conditions of posets in the classical setting provided the construction of the ordinal product is realized by a triangular norm operator with no zero divisors. The results obtained within this paper suggest that triangular norm operators with no zero divisors play an important role within the algebraic analysis of fuzzy posets and lattices.     

  The contents of this article are organized in the following manner: Section 2 establishes some preliminaries of bounded fuzzy lattices. In section 3, we exposit the basics of triangular norm operators and characterize the general triangular norm based direct product construction on bounded fuzzy lattices. In section 4, we prove that fuzzy posets are closed under direct products whenever the triangular norm realizing the product construction has no zero divisors and then extend this result to the case of bounded fuzzy lattices. Some immediate consequences of this result are then established. In section 5, we introduce a triangular norm based construction of ordinal products and investigate the basic properties of this construction.

\section{Preliminaries}

\begin{definition}\label{2.1}
    Let $X$ be a set, let $[0,1]=\{a\in\mathbb{R}:0\leq a\leq 1\}$, and let $\mu\colon X\times X\to[0,1]$ be a fuzzy relation. Then for all $x,y,z\in X$, we define the following relational properties:
    \begin{enumerate}
        \item $\mu$ is \emph{reflexive} if $\mu(x,x)=1$; \label{2.1.1}
        \item $\mu$ is \emph{transitive} if $\mu(x,y)>0$ and $\mu(y,z)>0$ implies $\mu(x,z)>0$; \label{2.1.2}\footnote{This formulation of fuzzy transitivity of $\mu$ is known as \emph{preference sensitive transitivity} and was studied by Barrett \emph{et. al.} in \cite{bar}. This may be contrasted to the more standard sup-min composition formulation of fuzzy transitivity characterized by the inequality $\mu(x,z)\geq\sup_{y\in X}\min(\mu(x,y),\mu(y,z))$ introduced by Zadeh in \cite{zad1}. Although Chon \cite{cho2009} and Mezzomo \emph{et al.} \cite{mez} use the latter, we use the former as it is a more natural generalization of the transitivity condition of crisp posets. For a survey of various formulations of fuzzy transitivity, we refer the reader to Dasgupta and Deb in \cite{das} for more details. 
        
        We note that although the results described in this section were originally proven for fuzzy posets defined using sup-min transitivity, they remain true in the setting of fuzzy posets defined using preference-sensitive transitivity.}
        \item $\mu$ is \emph{anti-symmetric} if $\mu(x,y)>0$ and $\mu(y,x)>0$ implies $x=y$; \label{2.1.3}
        \item $\mu$ is a \emph{fuzzy partial ordering} if $\mu$ is reflexive, transitive, and anti-symmetric.\label{2.1.4}   
    \end{enumerate}
\end{definition}

For any set $X$ and any fuzzy relation $\mu\colon X\times X\to[0,1]$, we call $\mathbb{X}=\langle X;\mu\rangle$ a \emph{fuzzy relational frame}. We call a fuzzy relational frame $\mathbb{X}=\langle X;\mu\rangle$ a \emph{fuzzy partially ordered set} (or simply, a \emph{fuzzy poset}) whenever $\mu$ is a fuzzy partial ordering on $X$. Moreover, if either $\mu(x,y)>0$ or $\mu(y,x)>0$ for all $x,y\in X$, we call $\mu$ a \emph{fuzzy linear order} on $X$ and $\mathbb{X}=\langle X;\mu\rangle$ a \emph{linear fuzzy poset}. 

\begin{definition}\label{2.2}
    Let $\mathbb{X}=\langle X;\mu\rangle$ be a fuzzy poset and let $Y\subseteq X$. Then: 
    \begin{enumerate}
        \item an element $x\in X$ is a \emph{fuzzy lower bound} of $Y$ if $\mu(x,y)>0$ for all $y\in Y$;
        \item an element $x\in X$ is a \emph{fuzzy upper bound} of $Y$ if $\mu(y,x)>0$ for all $y\in Y$;
        \item a fuzzy lower bound $x_0\in X$ of $Y$ is the \emph{greatest fuzzy lower bound} (or \emph{fuzzy meet}) of $Y$ if $\mu(x,x_0)>0$ for every fuzzy lower bound $x\in X$ of $Y$;
        \item a fuzzy upper bound $x_0\in X$ of $Y$ is the \emph{least fuzzy upper bound} (or \emph{fuzzy join}) of $Y$ if $\mu(x_0,x)>0$ for every fuzzy upper bound $x\in X$ of $Y$. 
    \end{enumerate}
\end{definition}  
Throughout the remainder of this work, given a fuzzy poset $\mathbb{X}=\langle X;\mu\rangle$, we denote the operations of fuzzy meet and fuzzy join for any $x,y\in X$ by $x\odot y$ and $x\oplus y$, respectively. Notice that $x\odot y$ and $x\oplus y$ are uniquely determined (whenever they exist) by the anti-symmetry of $\mu$.

\begin{definition}\label{2.3}
Let $\mathbb{X}=\langle X;\mu\rangle$ be a fuzzy poset. Then $\mathbb{X}$ is a \emph{fuzzy lattice} whenever $x\odot y$ and $x\oplus y$ exist for all $x,y\in X$. Moreover,  $\mathbb{X}$ is \emph{bounded} if there exist elements $w,z\in X$ such that $\mu(w,x)>0$ and $\mu(x,z)>0$ for all $x\in X$. In this case we call $w$ and $z$ the \emph{fuzzy bottom} and \emph{fuzzy top} elements of $\mathbb{X}$ respectively and occasionally will write $0_X:=w$ and $1_X:=z$.     
\end{definition}

\begin{example}\label{2.4}
    Table \ref{table1} depicts the corresponding matrices of two finite bounded fuzzy lattices.\footnote{Direct products of these specific lattices are given in Tables \ref{table2} and \ref{table3} under various triangular norm based realizations.} 
    \end{example}

\begin{table}[htbp]
\fontsize{4pt}{10pt}\selectfont
\centering
\begin{tabular}{|c|c|c|c|c|}
\hline
$\mu_{X_1}$ & $w_1$ & $x_1$ & $y_1$ & $z_1$ \\
 \hline
 $w_1$ & 1 & 0.1 & 0.4 & 0.8\\
  \hline
  $x_1$ & 0 & 1 & 0.2 & 0.5\\
   \hline
   $y_1$ & 0 & 0 & 1 & 0.3\\
    \hline
    $z_1$ & 0 & 0 & 0 & 1\\
    \hline
\end{tabular}
\hskip 5em
\begin{tabular}{|c|c|c|c|c|}
\hline
$\mu_{X_2}$ & $w_2$ & $x_2$ & $y_2$ & $z_2$ \\
 \hline
 $w_2$ & 1 & 0.1 & 0.3 & 0.9\\
  \hline
  $x_2$ & 0 & 1 & 0 & 0.6\\
   \hline
   $y_2$ & 0 & 0 & 1 & 0.4\\
    \hline
    $z_2$ & 0 & 0 & 0 & 1 \\
    \hline
\end{tabular}\caption{Bounded fuzzy lattices $\mathbb{X}_1=\langle X_1;\mu_{X_1}\rangle$ and $\mathbb{X}_2=\langle X_2,\mu_{X_2}\rangle$}\label{table1}
\end{table}
Note that $\mathbb{X}_1$ forms a linear fuzzy poset whereas $\mathbb{X}_2$ forms a non-linear fuzzy poset since: \[\mu_{X_2}(x_2,y_2)=\mu_{X_2}(y_2,x_2)=0\] 

\begin{proposition}[\cite{cho2009}]\label{prop2.5}
    Let $\mathbb{X}=\langle X;\mu\rangle$ be a fuzzy lattice. Then for all $x,y\in X$ we have:
    \begin{enumerate}
    \item $\mu(x,x\oplus y)>0$ and $\mu(y,x\oplus y)>0$ and $\mu(x\odot y,x)>0$ and $\mu(x\odot y,y)>0$
    \item $\mu(x,z)>0$ and $\mu(y,z)>0$ implies $\mu(x\oplus y,z)>0$
    \item $\mu(z,x)>0$ and $\mu(z,y)>0$ implies $\mu(z,x\odot y)>0$
    \item $\mu(x,y)>0$ if and only if $x\oplus y=y$\label{2.5(2)} 
    \item $\mu(x,y)>0$ if and only if $x\odot y=x$\label{2.5(1)}
         \item $\mu(y,z)>0$ implies $\mu(x\odot y,x\odot z)>0$ and $\mu(x\oplus y,x\oplus z)>0$
    \end{enumerate}
\end{proposition}

  From an algebraic perspective, Proposition \ref{2.6} together with Proposition \ref{2.7} show that Definition \ref{2.3} provides a suitable formulation of bounded fuzzy lattices in the sense that whenever $\mathbb{X}$ is a bounded fuzzy lattice, $\langle \mathbb{X};\odot,1_X\rangle$ and $\langle\mathbb{X};\oplus,0_X\rangle$ form idempotent commutative monoids connected by the absorption identities.    

\begin{proposition}[\cite{cho2009}]\label{2.6}
    For any bounded fuzzy lattice $\mathbb{X}=\langle X;\mu\rangle$ and any $x,y,z\in X$, we have:
        \begin{enumerate}
              \item\label{e2.61} $x\odot x=x$,\hspace{.1cm} $x\oplus x=x$ 
        \item\label{2.62} $x\odot y=y\odot x$,\hspace{.1cm} $x\oplus y=y\oplus x$ 
              \item\label{e2.63} $x=x\odot(x\oplus y)$,\hspace{.1cm} $x=x\oplus(x\odot y)$ 
        \item\label{e2.64} $x\odot(y\odot z)=(x\odot y)\odot z$,\hspace{.1cm} $x\oplus(y\oplus z)=(x\oplus y)\oplus z$ 
    \end{enumerate}  
\end{proposition}
The following proposition was not stated in \cite{cho2009, mez} but follows immediately from the definition of the bounds $0_X$ and $1_X$ as well as conditions 4 and 5 of Definition \ref{2.6}. 
\begin{proposition}\label{2.7}
    For any bounded fuzzy lattice $\mathbb{X}=\langle X;\mu\rangle$, $0_X$ is a unit with respect to $\odot$ and $1_X$ is a unit with respect to $\oplus$, i.e., $x\odot 1_X=x$ and $x\oplus 0_X=x$ for all $x\in X$
\end{proposition}

Given bounded fuzzy lattices $\mathbb{X}_1=\langle X_1;\mu_{X_1}\rangle$ and $\mathbb{X}_2=\langle X_2;\mu_{X_2}\rangle$, a function $\phi\colon\mathbb{X}_1\to\mathbb{X}_2$ is \emph{monotone} whenever $\mu_{X_1}(x_1,y_1)>0$ implies $\mu_{X_2}(\phi(x_1),\phi(y_1))>0$, and a \emph{bounded homomorphism} whenever $\phi$ preserves the lattice operations from $\mathbb{X}_1$ to $\mathbb{X}_2$. Lastly, $\phi$ is an \emph{isomorphism} if $\phi$ is a bijective bounded homomorphism.

    Clearly by Proposition \ref{prop2.5}, every bounded homomorphism between bounded fuzzy lattices is a monotone function. We note that the class of bounded fuzzy lattices along with the class of their associated bounded homomorphisms can be easily seen to form a category, which we will denote by $\mathbf{BFL}$ (consult Mac Lane \cite{mac} for more details on basic category theory). More generally, one can easily see that the class of fuzzy posets and their associated monotone functions form a category, which we shall denote by $\mathbf{FP}$.

\section{Direct minimum products}
Recall that if $\mathbb{X}_1=\langle X_1;\leq\rangle,\dots,\mathbb{X}_n=\langle X_n;\preceq\rangle$ are bounded lattices, then their \emph{direct product} is a relational structure $\prod^n_{i=1}\mathbb{X}_i=\langle \prod^n_{i=1}X_i;\leq\rangle$ such that $\prod^n_{i=1}X_i=\{(x_1,\dots,x_n):x_i\in X_i\hspace{.2cm}\text{for}\hspace{.2cm}1\leq i\leq n\}$ is the $n$-ary Cartesian product and $\leq$ is a partial ordering defined coordinate-wise by $(x_1,\dots,x_n)\leq(y_1,\dots,y_n)$ if and only if $x_i\leq_{X_i}y_i$ for $1\leq i\leq n$. Moreover, meets, joins, the bottom element, and the top element are defined coordinate-wise. It is well-known that bounded lattices are closed under the formation of direct products. 

Generalizing direct products to the setting of bounded fuzzy lattices relies on the class of triangular norm operators. Consult Klement \emph{et al.} \cite{kle} for more details on the basic theory of triangular norms.   

\begin{definition}\label{3.1}
    A \emph{triangular norm} is an operator of type $\tau\colon[0,1]^2\to[0,1]$ satisfying:
    \begin{multicols}{2}
    \begin{enumerate}\label{3.11}
        \item\label{3.1(1)} $\tau(a,\tau(b,c))=\tau(\tau(a, b), c)$
        \item\label{3.1(2)} $b\leq c\Rightarrow \tau(a,b)\leq \tau(a,c)$ 
        \item\label{3.1(3)} $\tau(a,b)=\tau(b,a)$ 
        \item\label{3.1(4)} $\tau(a,1)=a$
    \end{enumerate}
    \end{multicols}
\end{definition}
 The following is well-known and shows that every triangular norm coincides on the boundary of the unit square. 
\begin{proposition}[\cite{kle}]\label{p3.2}
    Every triangular norm $\tau\colon[0,1]^2\to[0,1]$ satisfies: 
    \begin{enumerate}
        \item $\tau(1,a)=a$
        \item $\tau(a,0)=\tau(0,a)=0$
    \end{enumerate}
\end{proposition}

Notice that by Definition \ref{3.1}(1), every triangular norm $\tau\colon[0,1]^2\to[0,1]$ can be uniquely extended to an $n$-ary operation on each $n$-tuple $(a_1,\dots,a_n)\in[0,1]^n$ by induction in the expected way:
\[\T^n_{i=1}a_i=\tau\biggl(\T^{n-1}_{i=1}a_i,a_n\biggl)\]

\begin{definition}\label{4.6}
    Let $\tau\colon [0,1]^2\to[0,1]$ be a triangular norm and let $(0,1)=\{r\in\mathbb{R}:0<r<1\}$. Then $a\in(0,1)$ is said to be a \emph{zero-divisor} for $\tau$ if there exists some $b\in(0,1)$ such that $\tau(a,b)=0$. In this case, we say that $\tau$ \emph{has zero divisors}.   
\end{definition}
A straightforward induction on $n$ shows that if $\tau$ is a triangular norm with zero divisors (respectively, no zero divisors), then every $n$-ary extension of $\tau$ has zero divisors (respectively, no zero divisors).  

\begin{example}\label{4.7}
   Lukasiewicz triangular norms obviously have zero divisors since for any $a,b\in(0,1)$ such that $a+b\leq 1$, we clearly have $\max\{a+b-1,0\}=0$. However the minimum, algebraic product, and Hamacher triangular norms have no zero divisors since for any $a,b\in(0,1)$, $\min\{a,b\}$, $a\cdot b$, $\frac{a\cdot b}{a+b-a\cdot b}>0$.  
\end{example}

\begin{definition}\label{3.3}
    Let $\mathbb{X}_1=\langle X_1;\mu_{X_1}\rangle,\dots,\mathbb{X}_n=\langle X_n;\mu_{X_n}\rangle$ be a finite family of bounded fuzzy lattices. Their \emph{direct product} is a fuzzy relational frame $\prod^n_{i=1}\mathbb{X}_i=\big\langle \prod^n_{i=1}X_i;\mu_p\big\rangle$ such that:
    \begin{enumerate}
        \item $\prod^n_{i=1}X_i$ is the $n$-ary Cartesian product;
        \item $\mu_p$ is a fuzzy relation of type $\mu_p\colon \prod^n_{i=1}X_i\times\prod^n_{i=1}X_i\to[0,1]$ defined by:
        \par
        \vspace{-.2cm}
        \[
        \mu_p((x_1,\dots,x_n),(y_1,\dots,y_n))=\T^n_{i=1}\mu_{X_i}(x_i,y_i)
      \]
 for some triangular norm $\tau$ where:
\[\T^n_{i=1}\mu_{X_i}(x_i,y_i)=\tau\biggl(\T^{n-1}_{i=1}\mu_{X_i}(x_i,y_i),\mu_{X_n}(x_n,y_n)\biggl)\]
        In this case, we say that the $n$-ary extension of $\tau$ \emph{realizes} the construction of $\prod^n_{i=1}\mathbb{X}_i$;  
        \item the fuzzy meet of any $\{(x_1,\dots,x_n),(y_1,\dots,y_n)\}\subseteq\prod^n_{i=1}X_i$ is defined coordinate-wise by:
        \[(x_1,\dots,x_n)\odot(y_1,\dots,y_n)=(x_1\odot y_1,\dots,x_n\odot y_n);\]
        \item the fuzzy join of any $\{(x_1,\dots,x_n),(y_1,\dots,y_n)\}\subseteq\prod^n_{i=1}X_i$ is defined coordinate-wise by:
        \[(x_1,\dots,x_n)\oplus(y_1,\dots,y_n)=(x_1\oplus y_1,\dots,x_n\oplus y_n);\]
        \item the fuzzy bottom and top elements of $\prod^n_{i=1}X_i$ are given by $(0_{X_1},\dots,0_{X_n})$ and $(1_{X_1},\dots,1_{X_n})$. 
    \end{enumerate}
\end{definition}

Notice that Definition \ref{3.3} is quite general in the sense that the fuzzy relational structure obtained by calculating the direct product is contingent upon the particular triangular norm one chooses as its realization. 

Chon \cite{cho2009} as well as Mezzomo, Bedregal, and Santiago \cite{mez} studied direct products realized by the minimum triangular norm, i.e., direct products whose fuzzy relational structure is calculated by:
\[\mu_p\colon\prod^n_{i=1}X_i\times\prod^n_{i=1}X_i\to[0,1];\hspace{.3cm}\mu_p((x_1,\dots,x_n),(y_1,\dots,y_n))=\min(\mu_{X_1}(x_1,y_1),\dots,\mu_{X_n}(x_n,y_n)),\] which we will refer to as the \emph{direct minimum product} of $\mathbb{X}_1,\dots,\mathbb{X}_n$. They obtained the following result.

\begin{theorem}[Chon \cite{cho2009}\label{3.5}, Mezzomo \emph{et al.} \cite{mez}]
    Given bounded fuzzy lattices $\mathbb{X}_1,\dots,\mathbb{X}_n$, their direct minimum product $\prod^n_{i=1}\mathbb{X}_i$ is a bounded fuzzy lattice. 
\end{theorem}

The proceeding example is a concrete instance of Theorem \ref{3.5} in the case when $n=2$. 
\begin{example}
Table \ref{table2} depicts the direct minimum product of the lattices $\mathbb{X}_1$ and $\mathbb{X}_2$ in Table \ref{table1}.  
\end{example}
\begin{table}[htbp]
\fontsize{4pt}{10pt}\selectfont
\begin{tabular}{|c|c|c|c|c|c|c|c|c|c|c|c|c|c|c|c|c|}
\hline
$\mu_p$ & $w_1w_2$ & $w_1x_2$ & $w_1y_2$ & $w_1z_2$ & $x_1w_2$ & $x_1x_2$ & $x_1y_2$ & $x_1z_2$ & $y_1w_2$ & $y_1x_2$ & $y_1y_2$ & $y_1z_2$ & $z_1w_2$ & $z_1x_2$ & $z_1y_2$ & $z_1z_2$ \\ \hline
$w_1w_2$ & 1     &   0.1    &   0.3     &   0.9     &   0.1     &   0.1     &   0.1     &  0.1      &    0.4    &  0.1      & 0.3       &     0.4   &   0.8     &   0.1      &  0.3      &  0.8      \\ \hline
$w_1x_2$ & 0  &     1   &    0    &   0.6     &   0     &   0.1     &    0    &    0.1    &     0   &   0.4     &    0    &    0.4    &   0     &     0.8    &    0    &     0.6   \\ \hline
$w_1y_2$ &  0  &      0  &   1     &   0.4     &    0    & 0 & 0.1       &   0.1     &     0   &    0    &   0.4     &   0.4     &    0    &     0    &    0.8    &    0.4    \\ \hline
$w_1z_2$ &  0  &   0     &     0   &   1     &   0     &   0     &   0     &   0.1     &    0    &   0     &    0    &    0.4    &     0   &    0     &    0    &   0.8     \\ \hline
$x_1w_2$ &  0  &    0    &     0   &   0    &   1     &    0.1    &    0.3    &     0.9   &   0.2     &    0.1    &   0.2     &   0.2     &   0.5     &     0.1    &   0.3     &    0.5    \\ \hline
$x_1x_2$ &  0  &     0   &     0   &   0     &   0     &    1    &    0    &     0.6   &   0     &    0.2    &   0     &   0.2     &    0    &   0.5      &      0  &     0.5   \\ \hline
$x_1y_2$ &  0  &   0     &     0   &   0    &    0    &     0   &    1    &    0.4    &  0      &    0    &   0.2     &   0.2     &    0    &     0    &     0.5   &   0.4     \\ \hline
$x_1z_2$ &  0  &    0    &     0   &   0     &   0     &    0    &   0     &   1     &     0   &   0     &     0   &   0.2     &   0     &      0   &  0      &   0.5     \\ \hline
$y_1w_2$ &  0  &     0   &     0   &   0    &    0    &     0   &    0    &    0    &  1      &   0.1     &    0.3    &   0.9     &     0.3   &    0.1     &  0.3      &    0.3    \\ \hline
$y_1x_2$ &  0  &      0  &     0   &   0    &    0    &     0   &    0    &    0    &   0     &   1     &   0     &   0.6     &     0   &    0.3     &     0   &    0.3    \\ \hline
$y_1y_2$ &  0  &       0 &     0   &   0    &    0    &     0   &    0    &    0    &   0     &   0    &   1     &  0.4      &     0   &    0     &      0.3  &     0.3   \\ \hline
$y_1z_2$ &  0  &        0 &    0    &  0     &   0     &    0    &   0     &   0     &  0      &  0      &     0   &  1      &      0  &     0    &   0     &    0.3    \\\hline
$z_1w_2$ &  0  &        0 &    0    &  0     &   0     &    0    &   0     &   0     &  0      &   0     &     0   &   0     &  1      &    0.1     &    0.3    &    0.9    \\\hline
$z_1x_2$ &  0  &        0 &    0    &  0     &   0     &    0    &   0     &   0    &   0     &     0   &      0  &     0   &     0   &  1       &     0   &   0.6     \\ \hline
$z_1y_2$ &  0  &        0 &    0    &  0     &   0     &    0    &   0     &   0     &  0      &     0   &     0   &     0   &   0     &   0      &   1     &    0.4    \\ \hline
$z_1z_2$ &  0  &        0 &    0    &  0     &   0     &    0    &   0     &   0     &  0      &     0   &    0    &     0   &  0      &    0     &   0     &  1      \\ \hline
\end{tabular}\caption{The direct minimum product $\mathbb{X}_1\times\mathbb{X}_2$}\label{table2}
\end{table}

\section{Direct products realized by \emph{t}-norms with no zero divisors}
We proceed by generalizing Theorem \ref{3.5} to the case of direct products realized by a much more general class of triangular norms.

The following trivial observation will be exploited later on\begin{proposition}\label{2.8}
    The fuzzy relational frame $\mathbb{X}=\langle X;\mu\rangle$ such that $X=\{x\}$ where $\mu$ is defined by $\mu(x,x)=1$ is a bounded fuzzy lattice. 
\end{proposition}

\begin{remark}\label{2.9}
  It is important to notice that the one-element bounded fuzzy lattice $\mathbb{X}$ described in Proposition \ref{2.8} is (up to isomorphism) the unique bounded fuzzy lattice satisfying the universal mapping property (in the sense of category theory) that for every bounded fuzzy lattice $\mathbb{Y}$, there exists a unique bounded homomorphism $\phi\colon\mathbb{Y}\to\mathbb{X}$. In other words, $\mathbb{X}$ is a \emph{terminal object} in the category \textbf{BFL}. More generally, since $\mathbb{X}$ is a fuzzy poset, $\mathbb{X}$ is also a terminal object in $\mathbf{FP}$.    
\end{remark}
\begin{lemma}[\cite{mac}]\label{4.1}
    A category $\mathscr{C}$ is closed under taking all finite products if and only if $\mathscr{C}$ has all binary products and a terminal object.
\end{lemma}
\begin{lemma}\label{4.2}
    If $\mathbb{X}_1=\langle X_1;\mu_{X_1}\rangle,\dots,\mathbb{X}_n=\langle X_n;\mu_{X_n}\rangle$ are fuzzy posets, then $\prod_{i=1}^n\mathbb{X}_i=\langle\prod^n_{i=1}X_i;\mu_p\rangle$ is a fuzzy poset whenever the triangular norm whose $n$-ary extension realizes its construction has no zero divisors.     
\end{lemma}
\begin{proof}
By Lemma \ref{4.1} and Remark \ref{2.9}, it suffices to show the case when $n=2$. Therefore, assume that $\mathbb{X}_1=\langle X_1;\mu_{X_1}\rangle$ and $\mathbb{X}_2=\langle X_2;\mu_{X_2}\rangle$ are fuzzy posets. We first demonstrate that $\mathbb{X}_1\times\mathbb{X}_2=\langle X_1\times X_2;\mu_p\rangle$ is a reflexive fuzzy relational frame by showing $\mu_p((x_1,x_2),(x_1,x_2))=1$ for all $x_1\in X_1$ and $x_2\in X_2$. By hypothesis, $\mathbb{X}_1$ and $\mathbb{X}_2$ are reflexive frames and hence: 
\begin{equation}\label{e4.2}
    \mu_{X_1}(x_1,x_1)=\mu_{X_2}(x_2,x_2)=1
\end{equation}
The calculation that $\mathbb{X}_1\times\mathbb{X}_2$ is a reflexive frame is then immediate: 
    \begin{align*}
        \mu_p((x_1,x_2),(x_1,x_2))&=\tau(\mu_{X_1}(x_1,x_1),\mu_{X_2}(x_2,x_2))\tag{by Definition \ref{3.3}(2)}
        \\&=\tau(1,1)=1\tag{by (\ref{e4.2}) and Definition \ref{3.1}(4)}
    \end{align*}

Now suppose that $\mathbb{X}_1$ and $\mathbb{X}_2$ are transitive frames. Then we have: 
\begin{equation}\label{ee4.6}
    \mu_{X_1}(x_1,y_1)>0\wedge\mu_{X_1}(y_1,z_1)>0\Rightarrow\mu_{X_1}(x_1,z_1)>0
\end{equation}
\begin{equation}\label{ee4.7}
    \mu_{X_2}(x_2,y_2)>0\wedge\mu_{X_2}(y_2,z_2)>0\Rightarrow\mu_{X_2}(x_2,z_2)>0
\end{equation}The proof that $\mathbb{X}_1\times\mathbb{X}_2$ is a transitive frame then runs as follows: 
\begin{align*}
    \mu_p((&x_1,x_2),(y_1,y_2))>0\wedge\mu_p((y_1,y_2),(z_1,z_2))>0\\&\Leftrightarrow\tau(\mu_{X_1}(x_1,y_1),\mu_{X_2}(x_2,y_2))>0\wedge\tau(\mu_{X_1}(y_1,z_1),\mu_{X_2}(y_2,z_2))>0 \tag{by Definition \ref{3.3}(2)}
    \\&\Leftrightarrow\mu_{X_1}(x_1,y_1)>0\wedge\mu_{X_1}(y_1,z_1)>0\wedge\mu_{X_2}(x_2,y_2)>0\wedge\mu_{X_2}(y_2,z_2)>0\tag{by Proposition \ref{p3.2}(2)}
    \\&\Rightarrow\mu_{X_1}(x_1,z_1)>0\wedge\mu_{X_2}(x_2,z_2)>0\tag{by (\ref{ee4.6}) and (\ref{ee4.7})}
    \\&\Leftrightarrow\tau(\mu_{X_1}(x_1,z_1),\mu_{X_2}(x_2,z_2))>0\tag{since $\tau$ has no zero divisors}
    \\&\Leftrightarrow\mu_p((x_1,x_2),(z_1,z_2))>0\tag{by Definition \ref{3.3}(2)}
\end{align*}

To see that $\mathbb{X}_1\times\mathbb{X}_2$ is an anti-symmetric frame, by Definition \ref{2.1}(3) it suffices to show:
\begin{equation}\label{e4.7}
    \mu_{p}((x_1,x_2),(y_1,y_2))>0\wedge\mu_{p}((y_1,y_2),(x_1,x_2))>0\Rightarrow(x_1,x_2)=(y_1,y_2)
\end{equation}
By hypothesis, $\mathbb{X}_1$ and $\mathbb{X}_2$ are anti-symmetric fuzzy relational frames i.e.:
    \begin{equation}\label{e4.8}
\mu_{X_1}(x_1,y_1)>0\wedge\mu_{X_1}(y_1,x_1)>0\Rightarrow x_1=y_1
\end{equation}
\begin{equation}\label{e4.9}
\mu_{X_2}(x_2,y_2)>0\wedge\mu_{X_2}(y_2,x_2)>0\Rightarrow x_2=y_2
    \end{equation}
 The result is achieved by observing the following: 
    \begin{align*}
        \mu_{p}&((x_1,x_2),(y_1,y_2))>0\wedge\mu_{p}((y_1,y_2),(x_1,x_2))>0
        \\&\Leftrightarrow \tau(\mu_{X_1}(x_1,y_1),\mu_{X_2}(x_2,y_2))>0\wedge\tau(\mu_{X_1}(y_1,x_1),\mu_{X_2}(y_2,x_2))>0\tag{by Definition \ref{3.3}(2)}        \\&\Rightarrow\mu_{X_1}(x_1,y_1)>0\wedge\mu_{X_1}(y_1,x_1)>0\wedge\mu_{X_2}(y_1,y_2)>0\wedge\mu_{X_2}(y_2,y_1)>0\tag{by Proposition \ref{p3.2}(2)}
    \end{align*}
By the above as well as (\ref{e4.8}) and (\ref{e4.9}), we have $x_1=y_1$ and $x_2=y_2$ and hence $(x_1,x_2)=(y_1,y_2)$. This establishes (\ref{e4.7}) and therefore we conclude that $\mathbb{X}_1\times\mathbb{X}_2$ is an anti-symmetric frame.   
\end{proof}
\begin{corollary}
    If $\mathbb{X}_1,\dots,\mathbb{X}_n$ are reflexive and anti-symmetric fuzzy relational frames, then $\prod^n_{i=1}\mathbb{X}_i$ is a reflexive and anti-symmetric fuzzy relational frame under the realization of any triangular norm. 
\end{corollary}
Notice however that Lemma \ref{4.2} will not in general be true if one drops the requirement that the triangular norm realizing the product construction has no zero divisors. For instance, direct Lukasiewicz products where:
   \vspace{-.2cm} \[\mu_p\colon\prod^n_{i=1}X_i\times\prod^n_{i=1}X_i\to[0,1];\hspace{.2cm}\mu_p((x_1,\dots,x_n),(y_1,\dots,y_n))=\max\Biggl(\sum^n_{i=1}\mu_{X_i}(x_i,y_i)-(n-1),0\Biggl)\] 
  are realized by a triangular norm with no zero divisors, i.e., the Lukasiewicz triangular norm. One can easily find fuzzy posets $\mathbb{X}_1,\dots,\mathbb{X}_n$ satisfying the following inequalities: \[\sum^n_{i=1}\mu_{X_i}(x_i,y_i)>n-1\hspace{.2cm}\text{and}\hspace{.2cm}\sum^n_{i=1}\mu_{X_i}(y_i,z_i)>n-1\hspace{.2cm}\text{but}\hspace{.2cm}\sum^n_{i=1}\mu_{X_i}(x_i,z_i)\leq n-1\] 
This implies that
$\mu_p((x_1,\dots,x_n),(y_1,\dots,y_n))>0$ as well as $\mu_p((y_1,\dots,y_n),(z_1,\dots,z_n))>0$ but:  
\[\mu_p((x_1,\dots,x_n),(z_1,\dots,z_n))=0\]
so that $\prod^n_{i=1}\mathbb{X}_i$ is not a transitive fuzzy relational frame and hence not a fuzzy poset.

Table 3 exemplifies the general construction given above in the case $n=2$. 

\begin{table}[htbp]
\fontsize{4pt}{10pt}\selectfont
\begin{tabular}{|c|c|c|c|c|c|c|c|c|c|c|c|c|c|c|c|c|}
\hline
$\mu_p$ & $w_1w_2$ & $w_1x_2$ & $w_1y_2$ & $w_1z_2$ & $x_1w_2$ & $x_1x_2$ & $x_1y_2$ & $x_1z_2$ & $y_1w_2$ & $y_1x_2$ & $y_1y_2$ & $y_1z_2$ & $z_1w_2$ & $z_1x_2$ & $z_1y_2$ & $z_1z_2$ \\ \hline
$w_1w_2$ & 1    & 0.1   &   0.3   &    0.9   &   0.1    & 0     &   0    &   0    &    0.4   &  0     &    0    &    0.3    &    0.8   &   0     &  0.1    &   0.7     \\ \hline
$w_1x_2$ & 0  &     1   &  0   &    0.6   &   0  &    0.1    &     0   &    0    &     0   &     0.4   &   0     &   0     &   0     &   0.8      &     0   &     0.4   \\ \hline
$w_1y_2$ &  0  &      0  &   1     &   0.4     &    0    &   0     &    0.1    &     0   &    0    &     0   &   0.4     &     0   &     0   &     0    &    0.8    &    0.2    \\ \hline
$w_1z_2$ &  0  &   0     &     0   &   1     &   0     &    0    &   0     &   0.1     &     0   &    0    &    0    &   0.4     &  0      &    0     &      0  &    0.8    \\ \hline
$x_1w_2$ &  0  &    0    &     0   &   0    &   1     &    0.1    &   0.3     &   0.9     &     0.2   &    0    &   0     &   0.1     &  0.5      &   0      &   0     &    0.4    \\ \hline
$x_1x_2$ &  0  &     0   &     0   &   0     &   0     &    1    &     0   &    0.6    &   0     &     0.2   &    0    &    0    &  0      &   0.5      &   0     &   0.1     \\ \hline
$x_1y_2$ &  0  &   0     &     0   &   0    &    0    &     0   &    1    &    0.4    &   0     &   0     &   0.2    &     0   &   0     &    0     &   0.5     &    0    \\ \hline
$x_1z_2$ &  0  &    0    &     0   &   0     &   0     &    0    &   0     &   1     &  0      &     0   &    0    &    0.2    &    0    &     0    &  0      &  0.5      \\ \hline
$y_1w_2$ &  0  &     0   &     0   &   0    &    0    &     0   &    0    &    0    &  1      &   0.1     &   0.3     &   0.9     &   0.3     &    0     &   0     &   0.2     \\ \hline
$y_1x_2$ &  0  &      0  &     0   &   0    &    0    &     0   &    0    &    0    &   0     &   1     &   0     &    0.6    &  0      &    0.3     &  0      &    0    \\ \hline
$y_1y_2$ &  0  &       0 &     0   &   0    &    0    &     0   &    0    &    0    &   0     &   0    &   1     &   0.4     &   0     &     0    &  0.3      &   0     \\ \hline
$y_1z_2$ &  0  &        0 &    0    &  0     &   0     &    0    &   0     &   0     &  0      &  0      &     0   &  1      &   0     &    0     &  0      &    0.3    \\\hline
$z_1w_2$ &  0  &        0 &    0    &  0     &   0     &    0    &   0     &   0     &  0      &   0     &     0   &   0     &  1      &    0.1     &   0.3     &    0.9    \\\hline
$z_1x_2$ &  0  &        0 &    0    &  0     &   0     &    0    &   0     &   0    &   0     &     0   &      0  &     0   &     0   &  1       &     0   &  0.6      \\ \hline
$z_1y_2$ &  0  &        0 &    0    &  0     &   0     &    0    &   0     &   0     &  0      &     0   &     0   &     0   &   0     &   0      &   1     &     0.4   \\ \hline
$z_1z_2$ &  0  &        0 &    0    &  0     &   0     &    0    &   0     &   0     &  0      &     0   &    0    &     0   &  0      &    0     &   0     &  1      \\ \hline
\end{tabular}\caption{The direct Lukasiewicz product $\mathbb{X}_1\times\mathbb{X}_2$}\label{table3}
\end{table}

It is clear that the direct Lukasiewicz product $\mathbb{X}_1\times\mathbb{X}_2$ depicted in Table 3 is not a transitive frame since for example we have $\mu_p((w_1,x_2),(x_1,x_2))>0$ and $\mu_p((x_1,x_2),(x_1,z_2))>0$ however $\mu_p((w_1,x_2),(x_1,z_2))=0$. Therefore, the hypothesis that the triangular norm realizing the direct product has no zero divisors is required in order to guarantee the preservation of transitivity.  We now arrive at the main result of this section.

\begin{theorem}\label{4.8}
    If $\mathbb{X}_1=\langle X_1;\mu_{X_1}\rangle,\dots,\mathbb{X}_n=\langle X_n;\mu_{X_n}\rangle$ is a finite family of bounded fuzzy lattices, then $\prod_{i=1}^n\mathbb{X}_i=\langle\prod^n_{i=1}X_i;\mu_p\rangle$ is a bounded fuzzy lattice whenever the triangular norm whose $n$-ary extension realizes its construction has no zero divisors.   
\end{theorem}
\begin{proof}
    By Lemma \ref{4.2}, we know what $\prod^n_{i=1}\mathbb{X}_i$ is a fuzzy poset. Again, it suffices to show the case when $n=2$. 
    
    Therefore let $\mathbb{X}_1=\langle X_1;\mu_{X_1}\rangle$ and $\mathbb{X}_2=\langle X_2;\mu_{X_2}\rangle$ be arbitrary bounded fuzzy lattices. We first want to show that $\mu_{p}((z_1,z_2),(x_1,x_2)\oplus(y_1,y_2))>0$ for $z_1\in\{x_1,y_1\}$ and $z_2\in\{x_2,y_2\}$. By hypothesis we have:
    \begin{equation}\label{e4.10}
        \mu_{X_1}(z_1,x_1\oplus y_1)>0,\hspace{.2cm}\mu_{X_2}(z_2,x_2\oplus y_2)>0
    \end{equation}

The calculation therefore proceeds as follows: 
    \begin{align*}
        \mu_{p}((z_1,z_2),(x_1,x_2)\oplus(y_1,y_2))&=\mu_{p}((z_1,z_2),(x_1\oplus y_1,x_2\oplus y_2))\tag{by Definition \ref{3.3}(4)}
        \\&=\tau(\mu_{X_1}(z_1,x_1\oplus y_1),\mu_{X_2}(z_2,x_2\oplus y_2))\tag{by Definition \ref{3.3}(2)}
        \\&>0\tag{by (\ref{e4.10}) and $\tau$ having no zero divisors}
    \end{align*}
It remains to show that $\mu_{p}((x_1,x_2)\oplus(y_1,y_2),(w_1,w_2))>0$ for any fuzzy upper bound $w_1$ of $\{x_1,y_1\}\subseteq X_1$ and any fuzzy upper bound $w_2$ of $\{x_2,y_2\}\subseteq X_2$. By hypothesis:
 \begin{equation}\label{e4.12}
     \mu_{X_1}(x_1\oplus y_1,w_1)>0,\hspace{.2cm}\mu_{X_2}(x_2\oplus y_2,w_2)>0
 \end{equation}
 The calculation runs similarly as above:
    \begin{align*}
        \mu_{p}((x_1,x_2)\oplus(y_1,y_2),(w_1,w_2))&=\mu_{p}((x_1\oplus y_1,x_2\oplus y_2),(w_1,w_2))\tag{by Definition \ref{3.3}(4)}
        \\&=\tau(\mu_{X_1}(x_1\oplus y_1,w_1),\mu_{X_2}(x_1\oplus y_2,w_2))\tag{by Definition \ref{3.3}(2)}
        \\&>0\tag{by (\ref{e4.12}) and $\tau$ having no zero divisors}
    \end{align*}
    Dually, we must now verify that for all $x_1,y_1\in X_1$ and $x_2,y_2\in X_2$ we have $\mu_{p}((x_1,x_2)\odot(y_1,y_2),(z_1,z_2))>0$ for $z_1\in\{x_1,y_1\}$ and $z_2\in\{x_2,y_2\}$. By hypothesis, $\mathbb{X}_1$ and $\mathbb{X}_2$ are fuzzy lattices and hence: 
\begin{equation}\label{e4.13}
    \mu_{X_1}(x_1\odot y_1,z_1)>0,\hspace{.2cm}\mu_{X_2}(x_2\odot y_2,z_2)>0
\end{equation}
The calculation runs dually to the case of $\oplus$:
\begin{align*}
    \mu_{p}((x_1,x_2)\odot(y_1,y_2),(z_1,z_2))&=\mu_{p}((x_1\odot y_1,x_2\odot y_2),(z_1,z_2))\tag{by Definition \ref{3.3}(3)}
    \\&=\tau(\mu_{X_1}(x_1\odot y_1,z_1),\mu_{X_2}(x_2\odot y_2,z_2))\tag{by Definition \ref{3.3}(2)}
    \\&>0\tag{by (\ref{e4.13}) and $\tau$ having no zero divisors}
\end{align*}
 It remains to show that $\mu_{p}((v_1,v_2),(x_1,x_2)\odot(y_1,y_2))>0$ for every fuzzy lower bound $v_1$ of $\{x_1,y_1\}\subseteq X_1$ and every fuzzy lower bound $v_2$ of $\{x_2,y_2\}\subseteq X_2$. By hypothesis, we have: 
\begin{equation}\label{e4.14}
    \mu_{X_1}(v_1,x_1\odot y_1)>0,\hspace{.2cm}\mu_{X_2}(v_2,x_2\odot y_2)>0
\end{equation}
The calculation runs similarly to the above case: 
\begin{align*}
     \mu_{p}((v_1,v_2),(x_1,x_2)\odot(y_1,y_2))&=\mu_{p}((v_1,v_2),(x_1\odot y_1,x_2\odot y_2))\tag{by Definition of \ref{3.3}(3)}
     \\&=\tau(\mu_{X_1}(v_1,x_1\odot y_1),\mu_{X_2}(v_2,x_2\odot y_2))\tag{by Definition \ref{3.3}(2)}
     \\&>0\tag{by (\ref{e4.14}) and $\tau$ having no zero divisors}
\end{align*}

It remains to verify that $\mathbb{X}_1\times\mathbb{X}_2$ is bounded by first showing that there exist $(0_{X_1},0_{X_2})$,$(1_{X_1},1_{X_2})\in X_1\times X_2$ such that $\mu_{p}((0_{X_1},0_{X_2}),(x_1,x_2))>0$ and $\mu_{p}((x_1,x_2),(1_{X_1},1_{X_2}))>0$ for all $(x_1,x_2)\in X_1\times X_2$. By hypothesis, $\mathbb{X}_1$ and $\mathbb{X}_2$ are bounded and hence there exist elements $0_{X_1}\in X_1$ and $0_{X_2}\in X_2$ such that $\mu_{X_1}(0_{X_1},x_1)>0$ for all $x_1\in X_1$ and $\mu_{X_2}(0_{X_2},x_2)>0$ for all $x_2\in X_2$. This, Definition \ref{3.3}(2), and lastly our assumption that $\tau$ has no zero divisors gives:
\begin{align*}
    \mu_{p}((0_{X_1},0_{X_2}),(x_1,x_2))&=\tau(\mu_{X_1}(0_{X_1},x_1),\mu_{X_2}(0_{X_2},x_2))\tag{by Definition \ref{3.3}(2)}
    \\&>0\tag{$0_{X_1}$, $0_{X_2}$ are bounds, $\tau$ having no zero divisors.}
\end{align*} 
 Similarly to the above case, we know that there exists $1_{X_1}\in X_1$ such that $\mu_{X_1}(x_1,1_{X_1})>0$ for all $x_1\in X_1$ and there exists $1_{X_2}\in X_2$ such that $\mu_{X_2}(x_2,1_{X_2})>0$ for all $x_2\in X_2$ since $\mathbb{X}_1$ and $\mathbb{X}_2$ are bounded. This, Definition \ref{3.3}(2), and our assumption that $\tau$ has no zero divisors gives:  
\begin{align*}
    \mu_{p}((x_1,x_2),(1_{X_1},1_{X_2}))&=\tau(\mu_{X_1}(x_1,1_{X_1}),\mu_{X_2}(x_2,1_{X_2}))\tag{by Definition \ref{3.3}(2)}
    \\&>0\tag{$1_{X_1}$, $1_{X_2}$ are bounds, $\tau$ having no zero divisors.}
\end{align*}
This completes the proof. 
 \end{proof}

 \begin{remark}
    Notice that Theorem \ref{4.8} is a significant generalization of Theorem \ref{3.5}. Whereas Theorem \ref{3.5} guarantees the closure of bounded fuzzy lattices under direct products realized specifically by the minimum triangular norm, Theorem \ref{4.8} guarantees the closure of bounded fuzzy lattices under a much more general construction of direct products; namely,  direct products realized by any triangular norm with no zero divisors, including the minimum triangular norm as a basic example of this more general result. Moreover, since the fuzzy lattices considered in this paper arise via fuzzy posets which are defined in part using preference sensitive transitivity, opposed to sup-min transitivity, we are working with a more general class of algebraic structures relative to those considered in \cite{cho2009, mez}.  
 \end{remark}

Further examples of Theorem \ref{4.8} include for instance the direct algebraic product as well as the Hamacher product, i.e., direct products with fuzzy relations $\mu_p\colon \prod^n_{i=1}X_i\times\prod^n_{i=1}X_i\to[0,1]$ defined by:
 \[\mu_{p}((x_1,\dots,x_n),(y_1,\dots,y_n))=\prod^n_{i=1}\mu_{X_i}(x_i,y_i)\]
  \[\mu_p((x_1,\dots,x_n),(y_1,\dots,y_n))=\begin{cases}0, & \text{if $\mu_{X_1}(x_1,y_1)=\dots=\mu_{X_n}(x_n,y_n)=0$}\\\frac{\prod^n_{i=1}\mu_{X_i}(x_i,y_i)}{\sum^n_{i=1}\mu_{X_i}(x_i,y_i)-\prod^n_{i=1}\mu_{X_i}(x_i,y_i)}, & \text{otherwise}\end{cases}\]
which as already described in Example \ref{4.7} are among the family of triangular norms with no zero divisors. 
  
  Closely related to the zero divisors of a triangular norm are its nilpotent elements.

\begin{definition}
    Let $\tau\colon[0,1]\times[0,1]\to[0,1]$ be a triangular norm. Then $a\in(0,1)$ is a \emph{nilpotent element} of $\tau$ if there exists some positive integer $n\in\mathbb{Z}^+$ such that $\T^n_{i=1}a_i=0$ where $a_i=a$ for $1\leq i\leq n$. 
\end{definition}

The following is well-known. 
\begin{lemma}[\cite{kle}]\label{4.11}
    For every triangular norm $\tau\colon[0,1]\times[0,1]\to[0,1]$, the following statements are equivalent: 
    \begin{enumerate}
        \item $\tau$ has zero divisors
        \item $\tau$ has nilpotent elements 
    \end{enumerate}
\end{lemma}
\begin{definition}
    Let $\tau\colon[0,1]\times[0,1]\to[0,1]$ be a triangular norm. Then:
    \begin{enumerate}
        \item $\tau$ is \emph{strictly monotone} whenever $x>0$ and $y<z$ implies $\tau(x,y)<\tau(x,z)$;
        \item $\tau$ is \emph{cancellative} whenever $\tau(x,y)=\tau(x,z)$ implies $x=0$ or $y=z$. 
    \end{enumerate}
\end{definition}

A straightforward induction on $n$ shows that if $\tau$ is a triangular norm with nilpotent elements (respectively, no nilpotent elements), then every $n$-ary extension of $\tau$ has nilpotent elements (respectively, no nilpotent elements) and similarly that if $\tau$ is strictly monotone (respectively, cancellative), then every $n$-ary extension of $\tau$ is strictly monotone (respectively, cancellative).  

\begin{lemma}[\cite{kle}]\label{strict}
    Let $\tau\colon[0,1]\times[0,1]\to[0,1]$ be a triangular norm. Then:
    \begin{enumerate}
        \item $\tau$ is strictly monotone if and only if $\tau$ is cancellative;
        \item if $\tau$ is strictly monotone, then $\tau$ has no zero divisors. 
    \end{enumerate}
\end{lemma}

By Lemmas \ref{4.11} and \ref{strict}, we immediately arrive at the following as a consequence of Theorem \ref{4.8}.

 \begin{corollary}
     Let $\mathbb{X}_1,\dots,\mathbb{X}_n$ be bounded fuzzy lattices. Then their direct product $\prod^n_{i=1}\mathbb{X}_i=\langle\prod^n_{i=1}X_i;\mu_p\rangle$ is a bounded fuzzy lattice whenever any of the following conditions obtain:
     \begin{enumerate}
         \item $\mu_p$ is defined by a triangular norm with no nilpotent elements;
         \item $\mu_p$ is defined by a strictly monotone triangular norm;
         \item $\mu_p$ is defined by cancellative triangular norm. 
     \end{enumerate}

 \end{corollary}

\section{Ordinal products realized by \emph{t}-norms with no zero divisors}

If $\mathbb{X}_1=\langle X_1;\leq_{X_1}\rangle$ and $\mathbb{X}_2=\langle X_2;\leq_{X_2}\rangle$ are posets, their \emph{ordinal product} is a relational structure $\mathbb{X}_1\otimes\mathbb{X}_2=\langle X_1\times X_2;\preceq_{X\times Y}\rangle$ such that $X_1\times X_2$ is their Cartesian product and $\preceq_{X\times Y}$ is a subset of $X_1\times X_2$, known as the \emph{lexicographic ordering}, defined by $(x_1,x_2)\preceq_{X_1\times X_2}(y_1,y_2)$ if and only either $x_1<_{X_1}y_1$ or $x_1=y_1$ and $x_2\leq_{X_2}y_2$. It is well-known that (linear) posets are closed under the formation of ordinal products. The aim of this section is to generalize this result to the setting of (linear) fuzzy posets.

The first step is to introduce the notion of ordinal products within the setting of fuzzy posets. 

\begin{definition}\label{ordinal product}
    Let $\mathbb{X}_1=\langle X_1;\mu_{X_1}\rangle$ and $\mathbb{X}_2=\langle X_2;\mu_{X_2}\rangle$ be fuzzy posets. Their \emph{ordinal product} is a fuzzy relational frame $\mathbb{X}_1\otimes\mathbb{X}_2=\langle X_1\times X_2;\mu_l\rangle$ such that: \begin{enumerate}
        \item $X_1\times X_2$ is the Cartesian product; 
        \item $\mu_l\colon(X_1\times X_2)\times(X_1\times X_2)\to[0,1]$ is a fuzzy relation defined by: 
        \[\mu_l((x_1,x_2),(y_1,y_2))=\begin{cases}
            \mu_{X_1}(x_1,y_1), & \text{if $\mu_{X_1}(x_1,y_1)>0$ and $x_1\not= y_1$}\\
            \tau(\mu_{X_1}(x_1,y_1),\mu_{X_2}(x_2,y_2)), & \text{otherwise}
        \end{cases}\]
        where $\tau\colon[0,1]^2\to[0,1]$ is some triangular norm. 
    \end{enumerate}   
\end{definition}

\begin{theorem}\label{thm ord prod}
    If $\mathbb{X}_1=\langle X_1,\mu_{X_1}\rangle$ and $\mathbb{X}_2=\langle X_2;\mu_{X_2}\rangle$ are fuzzy posets, then their ordinal product $\mathbb{X}_1\otimes\mathbb{X}_2$ is a fuzzy poset whenever it is realized by a triangular norm with no zero divisors. 
 \end{theorem}
 \begin{proof}
     Assume that $\mathbb{X}_1$ and $\mathbb{X}_2$ are fuzzy posets. Since $\mathbb{X}_1$ and $\mathbb{X}_2$ are reflexive frames, it is obvious that:
     \begin{equation}\label{5.1}
         \mu_{X_1}(x_1,x_1)=\mu_{X_2}(x_2,x_2)=1
     \end{equation}
     and therefore the following is immediate: 
     \begin{align*}
         \mu_l((x_1,x_2),(x_1,x_2))&=\tau(\mu_{X_1}(x_1,x_1),\mu_{X_2}(x_2,x_2))\tag{by Definition \ref{ordinal product}(2)}
         \\&=\tau(1,1)=1\tag{by (\ref{5.1}) and Definition \ref{3.1}(4)}
     \end{align*}
Therefore the ordinal product $\mathbb{X}_1\otimes\mathbb{X}_2$ is a reflexive frame.  

To see that $\mathbb{X}_1\otimes\mathbb{X}_2$ is anti-symmetric, assume $\mu_l((x_1,x_2),(y_1,y_2))>0$ and $\mu_l((y_1,y_2),(x_1,x_2))>0$. By hypothesis $\mathbb{X}_1=\langle X_1;\mu_{X_1}\rangle$ and $\mathbb{X}_2=\langle X_2;\mu_{X_2}\rangle$ are anti-symmetric and hence $\mu_{X_1}(x_1,y_1)>0$ and $\mu_{X_1}(y_1,x_1)>0$ implies $x_1=y_1$ as well as $\mu_{X_2}(x_2,y_2)>0$ and $\mu_{X_2}(y_2,x_2)>0$ implies $x_2=y_2$. By way of contradiction, suppose that $(x_1,x_2)\not=(y_1,y_2)$. Consider the case when $x_1\not=y_1$. By Definition \ref{2.1}(3) we have $\mu_{X_1}(x_1,y_1)=0$ or $\mu_{X_1}(y_1,x_1)=0$. If the former:
\begin{align*}
    \mu_{l}((x_1,x_2),(y_1,y_2))&=\tau(\mu_{X_1}(x_1,y_1),\mu_{X_2}(x_2,y_2))\tag{by Definition \ref{ordinal product}(2)}
    \\&=\tau(0,\mu_{X_2}(x_2,y_2))\tag{by hypothesis}
    \\&=0\tag{by Proposition \ref{p3.2}(2)}
\end{align*}     
which contradicts our hypothesis that $\mu_l((x_1,x_2),(y_1,y_2))>0$. The latter case when $\mu_{X_1}(y_1,x_1)=0$ runs the same.  If $x_2\not= y_2$, then again by Definition \ref{2.1}(3) we have $\mu_{X_2}(x_2,y_2)=0$ or $\mu_{X_2}(y_2,x_2)=0$ and the proofs follow similarly.  



 Now assume $\mathbb{X}_1$ and $\mathbb{X}_2$ are transitive frames and let $\mu_{l}((x_1,x_2),(y_1,y_2))>0$ and $\mu_{l}((y_1,y_2),(z_1,z_2))>0$. By Definition \ref{ordinal product}(2), it is then obvious that $\mu_{X_1}(x_1,y_1)$, $\mu_{X_1}(y_1,z_1)$, $\mu_{X_2}(x_2,y_2)$, $\mu_{X_2}(y_2,z_2)>0$. Since $\mathbb{X}_1=\langle X_1;\mu_{X_1}\rangle$ and $\mathbb{X}_2=\langle X_2;\mu_{X_2}\rangle$ are transitive, we find that $\mu_{X_1}(x_1,z_1)>0$ and $\mu_{X_2}(x_2,z_2)>0$. In the case when $x_1\not=z_1$, we have $\mu_{l}((x_1,x_2),(z_2,z_2))=\mu_{X_1}(x_1,z_1)>0$. In the case when $x_1=z_1$, by Definition \ref{ordinal product} and our assumption that the triangular norm $\tau$ realizing the ordinal product construction of $\mathbb{X}_1\otimes\mathbb{X}_2$ has no zero divisors, we have
    $\mu_{l}((x_1,x_2),(z_1,z_2))=\tau(\mu_{X_1}(x_1,z_1),\mu_{X_2}(x_2,z_2))>0$.
Therefore we conclude that $\mathbb{X}_1\otimes\mathbb{X}_2$ is a transitive frame and moreover, that $\mathbb{X}_1\otimes\mathbb{X}_2$ is a fuzzy poset, which completes the proof.  
 \end{proof}

A quick inspection of the proof of Theorem \ref{thm ord prod} yields the following. 
 
\begin{corollary}
    If $\mathbb{X}_1$ and $\mathbb{X}_2$ are reflexive and anti-symmetric fuzzy relational frames, then so is their ordinal product $\mathbb{X}_1\otimes\mathbb{X}_2$ under the realization of any triangular norm. 
\end{corollary}

The final result extends Theorem \ref{thm ord prod} to the case of linear fuzzy posets. 
 \begin{theorem}\label{ord prod poset}
     If $\mathbb{X}_1=\langle X_1;\mu_{X_1}\rangle$ and $\mathbb{X}_2=\langle X_2;\mu_{X_2}\rangle$ are linear fuzzy posets, then their ordinal product $\mathbb{X}_1\otimes\mathbb{X}_2$ is a linear fuzzy poset whenever it is realized by a triangular norm with no zero divisors. 
 \end{theorem}
 \begin{proof}
     Assume that $\mathbb{X}_1=\langle X_1;\mu_{X_1}\rangle$ and $\mathbb{X}_2=\langle X_2;\mu_{X_2}\rangle$ are linear fuzzy posets so that $\mu_{X_1}(x_1,y_1)>0$ or $\mu_{X_1}(y_1,x_1)>0$ for all $x_1,y_1\in X_1$ and $\mu_{X_2}(x_2,y_2)>0$ or $\mu_{X_2}(y_2,x_2)>0$ for all $x_2,y_2\in X_2$.


      If $\mu_{X_1}(x_1,y_1)>0$ such that $x_1\not=y_1$ and $\mu_{X_2}(x_2,y_2)>0$, then:
     \begin{align*}
         \mu_l((x_1,x_2),(y_1,y_2))&=\mu_{X_1}(x_1,y_1)\tag{by Definition \ref{ordinal product}(2)}
         \\&>0\tag{by hypothesis}
     \end{align*}
     If $\mu_{X_1}(x_1,y_1)>0$ such that $x_1=y_1$ and $\mu_{X_2}(x_2,y_2)>0$, then:
\begin{align*}
    \mu_l((x_1,x_2),(y_1,y_2))&=\tau(\mu_{X_1}(x_1,y_1),\mu_{X_2}(x_2,y_2))\tag{by Definition \ref{ordinal product}(2)}
    \\&=\tau(1,\mu_{X_2}(x_2,y_2))\tag{$\mu_{X_1}$ is reflexive and $x_1=y_1$}
    \\&=\mu_{X_2}(x_2,y_2)\tag{by Proposition \ref{p3.2}(1)}
    \\&>0\tag{by hypothesis}
\end{align*}

     The remaining cases run completely analogously and hence we omit them. Therefore, we conclude that $\mathbb{X}_1\otimes\mathbb{X}_2$ is a linear fuzzy poset whenever $\mathbb{X}_1$ and $\mathbb{X}_2$ are linear fuzzy posets.    
     \end{proof}
\begin{remark}
It is obvious within the above proof that the preservation of the property of linearity under the construction of ordinal products is independent of the triangular norm realizing the product construction. However, as was seen in the proof of Theorem \ref{thm ord prod}, the absence of zero divisors for the triangular norm realizing the ordinal product is required in preserving transitivity, hence the additional hypothesis regarding zero divisors within the statement of Theorem \ref{ord prod poset}.    
\end{remark}    

\section{Conclusion}
We have strengthened the results obtained within the setting of direct minimum products in \cite{cho2009, mez} by showing that bounded fuzzy lattices are closed under a much more general construction of direct products; namely direct products realized by triangular norms with no zero divisors. This result provides a purely algebraic understanding of the results in \cite{cho2009, mez} in terms of what general algebraic properties about the minimum triangular norm are being exploited within their proofs and in this sense, explains why their result is true in the case of direct minimum products. Moreover, since the fuzzy lattices considered in this paper arise via fuzzy posets which are defined in part using preference sensitive transitivity, opposed to sup-min transitivity, we are working with a more general class of algebraic structures relative to those considered in \cite{cho2009, mez}. A methodological advantage gained from the overall approach developed in this paper is that it allows for the application of a much more broad class of triangular norms within the theory of algebraic operations on fuzzy posets and lattices. We have also introduced ordinal products within the setting of fuzzy posets and shown this construction preserves the desired relational properties under the hypothesis that the triangular norm realizing the product construction has no zero divisors. Since closure under the formation of direct products (in the case of bounded fuzzy lattices) and ordinal products (in the case of (linear) fuzzy posets) are desirable properties, at least to the extent that they mimic classical theory, the results obtained in this work imply that triangular norms with no zero divisors play an important role within the algebraic analysis of fuzzy ordered sets.

The results obtained in this work also suggest that there are possible applications of the family of triangular norms with no zero divisors within the theory of fuzzy preference relations and social welfare functions (see \cite{bar, das}), wherein fuzzy transitive relations play an important role. Specifically, in approaches to these research areas in which preference sensitive transitivity is assumed. Recall that in the case of direct products as well as ordinal products of fuzzy posets, while the preservation of reflexivity and anti-symmetry are independent of the triangular norm one chooses in the construction of the product, this is not the case when it comes to the preservation of preference sensitive transitivity. Hence, this approach develops a method whereby one can perform various algebraic operations on fuzzy preference relations such that the essential property of preference sensitive transitivity is preserved under such constructions. Of course, additional hypotheses may need to be imposed on the underlying triangular norm to guarantee the preservation of additional fuzzy relational properties one may wish to impose of certain classes fuzzy preference relations.

\section*{Declaration of competing interest}
The author declares that he has no known competing financial interests or personal relationships that could have
appeared to influence the work reported in this paper.

\section*{Acknowledgments}
This research was funded under the CGS-D SSHRC grant no. 767-2022-1514. The author would like to thank Dr$.$\hspace{.05cm}Witold Pedrycz, the Fuzzy Set Theory Research Group in the Department of Electrical and Computer Engineering at the University of Alberta, as well as the participants of the 2023 BLAST conference (hosted by the Department of Mathematics and Statistics, University of North Carolina at Charlotte), and the 2024 UTEP-NMSU Workshop on Mathematics, Computer Science, and Computational Sciences (hosted by the Department of Computer Science, University of Texas at El Paso) for their useful comments and suggestions.


\begin{thebibliography}{1}
    \setlength{\itemsep}{0em}
    \setlength{\parskip}{0em}
\bibitem{ajm}Ajmal, N., Thomas, K.: Fuzzy lattices. \emph{Inf. Sci.} \textbf{79}, 271 -- 291 (1994)


\bibitem{bar} Barrett, C., Pattanaik, P., Salles, M.: On the structure of fuzzy social welfare functions. \emph{Fuzzy Sets Syst.} \textbf{19}, 1 -- 11 (1986)


\bibitem{cho2009} Chon, I.: Fuzzy partial order relations and fuzzy lattices. \emph{Korean J. Math.} \textbf{17}, 361--374 (2009)

\bibitem{das} Dasgupta, M., Deb, R.: Transitivity and fuzzy preferences. \emph{Soc. Choice Welf}. \textbf{13}, 305 -- 318 (1996)

\bibitem{dav} Davey, B., Priestley, H.: Introduction to Lattices and Order (second edition). Cambridge University Press, Cambridge (2002)



 



\bibitem{kle} Klement, E., Mesiar, R., Pap, E.: Triangular Norms. Trends in Logic, vol. \textbf{8}, Springer (2000) 



\bibitem{kur} Kuroki, N.: On fuzzy semigroups. \emph{Fuzzy Sets Syst.} \textbf{52}, 203 -- 236 (1991)

\bibitem{liu} Liu, W.J.: Fuzzy invariant subgroups and Fuzzy ideals. \emph{Fuzzy Sets Syst.} \textbf{8}, 133 -- 139 (1982)

\bibitem{mac} Mac Lane, S.: Categories for the Working Mathematician. Graduate Texts in Mathematics, vol. \textbf{5}, Springer-Verlag, New York (1971)

\bibitem{mez} Mezzomo, I., Bedregal, B., and Santiago, R.: On some operations on bounded fuzzy lattices. \emph{J. Fuzzy Math.} \textbf{22}, 853 -- 878 (2014)

\bibitem{ros} Rosenfeld, A.: Fuzzy groups. \emph{J. Math. Anal. Appl.} \textbf{35}, 5-12 -- 517 (1971)




\bibitem{yua} Yuan, B., Wu. W.: Fuzzy ideals on a distributive lattice. \emph{Fuzzy Sets Syst.} 
\textbf{35}, 231 -- 240 (1990) 



\bibitem{zad1} Zadeh, L.: Similarity relations and fuzzy orderings. \emph{Inf. Sci.}, \textbf{3}, 177 -- 200 (1971)
\end{thebibliography}
\end{document}